\documentclass[11pt]{article}
\usepackage{amsfonts}
\usepackage{latexsym}
\usepackage{cite}
\usepackage{amsmath,amsfonts,latexsym,amssymb}
\usepackage[mathscr]{eucal}
\usepackage{cases}
\usepackage{amsthm}

\usepackage[bf,small]{caption2}
\usepackage{float,color}
\usepackage{graphicx}
\usepackage{amsmath,url}
\usepackage{amssymb}
\usepackage[all]{xy}

\newtheorem{theorem}{theorem}[section]
\newtheorem{thm}[theorem]{Theorem}
\newtheorem{lem}[theorem]{Lemma}

\newtheorem{cor}[theorem]{Corollary}

\newtheorem{nota}[theorem]{Notation}

\newtheorem{rmk}[theorem]{Remark}

\begin{document}

\title{\textbf{Character varieties of even classical pretzel knots}}
\author{\Large Haimiao Chen
\footnote{Email: {\em\small chenhm@math.pku.edu.cn}} \\
\normalsize \em{Mathematics, Beijing Technology and Business University, Beijing, China}}

\date{}
\maketitle

\begin{abstract}
  For each even classical pretzel knot $P(2k_1+1,2k_2+1,2k_3)$, we determine the character variety of irreducible ${\rm SL}(2,\mathbb{C})$-representations, and clarify the steps of computing its A-polynomial.

\end{abstract}

\section{Introduction}

For a knot $K\subset S^{3}$, let $E_{K}=S^3-N(K)$, with $N(K)$ a tubular neighborhood.
The ${\rm SL}(2,\mathbb{C})$-{\it representation variety} of $K$ is the set $\mathcal{R}(K)$ consisting of representations $\rho:\pi_1(E_{K})\to{\rm SL}(2,\mathbb{C})$, and the {\it character variety} of $K$ is
$\mathcal{X}(K)=\{\chi(\rho)\colon\rho\in\mathcal{R}(K)\},$
where the {\it character} of $\rho$ is the function $\chi(\rho):\pi_1(E_{K})\to\mathbb{C}$ sending $x\in\pi_1(E_{K})$ to ${\rm tr}(\rho(x))$.

The reason for calling $\mathcal{R}(K)$ and $\mathcal{X}(K)$ varieties is that they can be defined by a finite set of polynomial equations \cite{CS83}.
Denote the subset of $\mathcal{R}(K)$ consisting of irreducible representations by $\mathcal{R}^{\rm irr}(K)$, then up to conjugacy, each $\rho\in\mathcal{R}^{\rm irr}(K)$ is determined by $\chi(\rho)$.
We mainly focus on $\mathcal{R}^{\rm irr}(K)$ and
$$\mathcal{X}^{\rm irr}(K):=\{\chi(\rho)\colon\rho\in\mathcal{R}^{\rm irr}(K)\};$$
reducible representations and their characters are easy to understand.

As seen in the literature, there seems to be difficulty in dealing with the representation/character variety of $K$ when $\pi_1(E_K)$ is generated by at least three generators. Although an effective algorithm for finding the character variety of any finitely presented group has been developed in \cite{ABL18}, systematically computing for a family at a time is another story. In \cite{Ch18-2}, the author computed the character varieties for classical odd pretzel knots, which form a 3-parameter family.

In this paper, we contribute one more piece, by determining the irreducible character variety for each even classical pretzel knot:
\begin{thm} \label{thm:main}
The irreducible character variety of the even pretzel knot $P(2k_1+1,2k_1+1,2k_3)$ can be embedded in
$$\{(t,s_1,s_2,s_3,\tau)\in\mathbb{C}^5\colon \tau^2-t(\sigma_1+2)\tau+t^2(\sigma_2+4)=4+\sigma_3+2\sigma_2-\sigma_1^2\},$$
and is the disjoint union of four parts:
$\mathcal{X}^{\rm irr}(K)=\mathcal{X}_0\sqcup\mathcal{X}_1\sqcup\mathcal{X}_2\sqcup\mathcal{X}_3,$
where
\begin{itemize}
  \item $\mathcal{X}_0=\mathcal{X}_{0,1}\sqcup\mathcal{X}_{0,2}$, where $\mathcal{X}_{0,1}$ consists of $(0,s_1,s_2,s_3,\tau)$ with
        $$\tau\ne 0,\qquad \gamma_1=-\beta_1, \qquad \gamma_2=-\beta_2, \qquad \beta_3=0,$$
        and $\mathcal{X}_{0,2}$ consists of $(0,2\cos\theta_1,2\cos\theta_2,2\cos\theta_3,0)$ with
        $$\cos(2k_1+1)\theta_1=\cos(2k_2+1)\theta_2=\cos(2k_3\theta_3)\ne -1;$$
  \item $\mathcal{X}_1=\mathcal{X}_{1,1}\sqcup\mathcal{X}_{1,2}\sqcup\mathcal{X}_{1,3}$, where $\mathcal{X}_{1,3}$ consists of $(\pm 2,s_1,s_2,2,\tau)$ with $\gamma_{1}=\beta_{1}$, $\gamma_2=\beta_2$,
        and for $j=1,2$, $\mathcal{X}_{1,j}$ consists of $(t,s_1,s_2,s_3,\tau)$ with
        $$\gamma_{j\pm}=\beta_{j\pm}, \qquad  t^2=s_{j}+2=s_{j+}+s_{j-};$$
  \item $\mathcal{X}_2$ consists of
        $$\Big(t,2\cos\frac{(2h_1+1)\pi}{2k_1+1},2\cos\frac{(2h_2+1)\pi}{2k_2+1},2\cos\frac{h_3\pi}{k_3},\tau\Big)$$
        with $h_1\in\{0,\ldots,k_1\}$, $h_2\in\{0,\ldots,k_2\}$, $h_3\in\{0,\ldots,k_3-1\},$
        so $\mathcal{X}_2$ is made up of $(k_1+1)(k_2+1)k_3$ conics;
  \item $\mathcal{X}_3$ consists of $(t,s_1,s_2,s_3,t\lambda)$ with
        \begin{align*}
        \sigma_1+2-2\lambda&\ne 0, \qquad t\ne 0, \\
        (\lambda-2-s_j)\gamma_j&=(\sigma_1-s_j-\lambda)\beta_j, \qquad j=1,2,  \\
        (\sigma_1+2-2\lambda)\alpha_3&=(s_3^2-s_3\lambda+\sigma_1-2)\beta_3.
        \end{align*}
\end{itemize}
The dimensions are: $\dim\mathcal{X}_{0}=\dim\mathcal{X}_{1}=0$, $\dim\mathcal{X}_2=\dim\mathcal{X}_3=1$.
\end{thm}
The proof is given in Section 3. Based on this, in Section 4 we also present a method for computing the A-polynomial.
The main line is parallel to that of \cite{Ch18-2}, but now we simplify some key steps, and fix a few newly arising issues.

\medskip

{\bf Acknowledgement} \\
The author is supported by NSFC-11771042.

\section{Preliminary}

For most part of this section, refer to \cite{Ch18-2} Section 2.

Let $\mathcal{M}(2,\mathbb{C})$ denote set of $2\times 2$ matrices with entries in $\mathbb{C}$; it is a 4-dimensional vector space over $\mathbb{C}$.
Let $\mathcal{U}(2,\mathbb{C})\subset\mathcal{M}(2,\mathbb{C})$ denote the subspace of upper-triangular matrices, and let $\mathcal{UT}=\mathcal{U}(2,\mathbb{C})\cap{\rm SL}(2,\mathbb{C})$.
Let $I$ denote the $2\times 2$ identity matrix.

Given $t\in \mathbb{C}$ and $k\in\mathbb{Z}$, take $a$ with $a+a^{-1}=t$ and put
\begin{align*}
\omega_{k}(t)=\begin{cases}
(a^{k}-a^{-k})/(a-a^{-1}), &a\notin\{\pm 1\}, \\
ka^{k-1}, &a\in\{\pm 1\};
\end{cases}
\end{align*}
note that the right-hand-side is unchanged when $a$ is replaced by $a^{-1}$. It is easy to verify that for all $k\in\mathbb{Z}$,
\begin{align}
\omega_k(t)+\omega_{-k}(t)&=0, \label{eq:omega0} \\
\omega_{k+1}(t)-t\omega_{k}(t)+\omega_{k-1}(t)&=0, \label{eq:omega1} \\
\omega_{k}(t)^{2}-t\omega_{k}(t)\omega_{k-1}(t)+\omega_{k-1}(t)^{2}&=1.  \label{eq:omega2}
\end{align}

If $X\in{\rm SL}(2,\mathbb{C})$ with ${\rm tr}(X)=t$, then repeated applications of Cayley-Hamilton Theorem leads to
\begin{align}
X^{k}=\omega_{k}(t)X-\omega_{k-1}(t)I \label{eq:X^k}
\end{align}
for all $k\in\mathbb{Z}$; in particular,
\begin{align}
X^{-1}=tI-X. \label{eq:X-inverse}
\end{align}

\begin{lem}  \label{lem:basic}
For any $X,Y\in{\rm SL}(2,\mathbb{C})$ with ${\rm tr}(X)=t_{1}$, ${\rm tr}(Y)=t_{2}$ and ${\rm tr}(XY)=t_{12}$, one has
\begin{align}
XYX&=t_{12}X-Y^{-1}, \label{eq:basic-1}  \\
XY+YX&=(t_{12}-t_{1}t_{2})I+t_{2}X+t_{1}Y.   \label{eq:basic-2}
\end{align}
\end{lem}

We call a set $\{X_1,\ldots,X_r\}$ {\it regular} if $X_1,\ldots,X_r$ do not have a common eigenvector.
The following results are well-known; one can refer to \cite{Go09} Section 2 and Section 5.
\begin{lem} \label{lem:non-regular}
For any $X,Y\in{\rm SL}(2,\mathbb{C})-\{\pm I\}$, the following conditions are equivalent to each other:
\begin{enumerate}
  \item[\rm(i)] $\{X,Y\}$ is not regular;
  \item[\rm(ii)] there exists $Z\in{\rm SL}(2,\mathbb{C})$ such that $ZXZ^{-1},ZYZ^{-1}\in\mathcal{UT}$;
  \item[\rm(iii)] $I,X,Y,XY$ are linear dependent as elements of $\mathcal{M}(2,\mathbb{C})$.
\end{enumerate}
\end{lem}

\begin{lem} \label{lem:3 matrix}
Given $t, t_{12}, t_{23}, t_{13}, t_{123}\in\mathbb{C}$, let
\begin{align*}
\nu_0&=t^2(3-t_{13}-t_{23}-t_{13})+t_{12}^2+t_{23}^2+t_{13}^2+t_{12}t_{23}t_{13}-4, \\
\nu_1&=t(t_{12}+t_{23}+t_{13})-t^3.
\end{align*}

{\rm(i)} There exist $X_1, X_2, X_3\in{\rm SL}(2,\mathbb{C})$ with ${\rm tr}(X_i)=t$, ${\rm tr}(X_{i}X_{j})=t_{ij}$ for $1\le i<j\le 3$ and ${\rm tr}(X_1X_2X_3)=t_{123}$ if and only if
\begin{align}
t_{123}^2-\nu_1t_{123}+\nu_0=0.   \label{eq:r-gamma}
\end{align}

{\rm(ii)} If {\rm(\ref{eq:r-gamma})} holds and $\{X_1,X_2,X_3\}$ is required to be regular, then the ordered triple $(X_1,X_2,X_3)$ is unique up to simultaneous conjugacy.
\end{lem}

\section{Irreducible representations}

In this section, we prove Theorem \ref{thm:main} through several lemmas and a corollary.

\begin{figure} [h]
  \centering
  \includegraphics[width=0.5\textwidth]{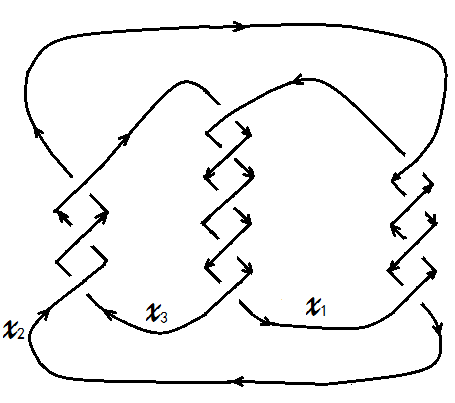}\\
  \caption{The pretzel knot $P(2k_1+1,2k_2+1,2k_3)$, with $k_1=1, k_2=k_3=2$}\label{fig:pretzel-even}
\end{figure}

Fix $K=P(2k_1+1,2k_2+1,2k_3)$ throughout this section.
Let $x_1,x_2,x_3\in\pi_1(E_K)$ be the elements represented by the arcs shown in Figure \ref{fig:pretzel-even}.

Similarly as in \cite{Ch18-2} Section 3.1, we can show: given $X_1,X_2,X_3\in{\rm SL}(2,\mathbb{C})$, there exists a representation $\rho:\pi_1(E_K)\to{\rm SL}(2,\mathbb{C})$ with
\begin{align*}
X_1=\rho(x_1), \qquad X_2=\rho(x_2), \qquad X_3=\rho(x_3^{-1})
\end{align*}
if and only if
\begin{align}
(X_2X_3^{-1})^{k_1}X_2(X_2X_3^{-1})^{-k_1}&=(X_3X_1^{-1})^{k_2+1}X_1(X_3X_1^{-1})^{-(k_2+1)},  \label{eq:relation1} \\
(X_3X_1^{-1})^{k_2}X_3(X_3X_1^{-1})^{-k_2}&=(X_1X_2^{-1})^{k_3}X_1^{-1}(X_1X_2^{-1})^{-k_3}, \label{eq:relation2} \\
(X_1X_2^{-1})^{k_3}X_2(X_1X_2^{-1})^{-k_3}&=(X_2X_3^{-1})^{k_1+1}X_3^{-1}(X_2X_3^{-1})^{-(k_1+1)}.  \label{eq:relation3}
\end{align}

\begin{nota}
\rm To simplify the writing, by $X_{j+}$ we mean $X_{j+1}$ for $j\in\{1,2\}$ and $X_1$ for $j=3$; by $X_{j-}$ we mean $X_{j-1}$ for $j\in\{2,3\}$ and $X_3$ for $j=1$. Similarly for other situations.
\end{nota}

Put
\begin{align*}
Y_j&=X_{j+}X_{j-}^{-1}, \qquad j=1,2,3, \\
A_j&=Y_j^{k_j}X_{j+}Y_j^{-k_j}X_{j-}, \qquad j=1,2, \\
A_3&=Y_3^{k_3}X_1^{-1}Y_3^{-k_3}X_1=Y_3^{k_3}X_2^{-1}Y_3^{-k_3}X_2.
\end{align*}
Then (\ref{eq:relation1})--(\ref{eq:relation3}) are equivalent to
\begin{align}
A_1=A_2=A_3. \label{eq:A}
\end{align}

We only consider irreducible representations, so $\{X_1,X_2,X_3\}$ is assumed to be regular.

Suppose
\begin{align*}
{\rm tr}(X_{1}X_{2}X_{3})&=r; \qquad {\rm tr}(X_{1})={\rm tr}(X_{2})={\rm tr}(X_{3})=t=u+u^{-1}, \\
&{\rm tr}(Y_{j})=s_j=v_j+v_j^{-1}, \qquad  j=1,2,3.
\end{align*}
Clearly,
\begin{align*}
{\rm tr}(X_{j+}X_{j-})={\rm tr}(X_{j+}(tI-X_{j-}^{-1}))=t^2-s_j.
\end{align*}

Let
\begin{align*}
\sigma_1=s_1+s_2+s_3,  \qquad \sigma_2&=s_1s_2+s_2s_3+s_3s_1,  \qquad \sigma_3=s_1s_2s_3, \\
\tau&=t^3+t-r, \\
\delta&=4+\sigma_3+2\sigma_2-\sigma_1^2, \\
\kappa&=\tau^2-t(\sigma_1+2)\tau+t^2(\sigma_2+4).
\end{align*}
Then (\ref{eq:r-gamma}) can be rewritten as
\begin{align}
\kappa=\delta. \label{eq:kappa=delta}
\end{align}

For each $j$, denote
\begin{align*}
\alpha_j=\omega_{k_j-1}(s_j), \qquad \beta_{j}=\omega_{k_j}(s_{j}), \qquad \gamma_{j}=\omega_{k_{j}+1}(s_{j}).
\end{align*}
Using (\ref{eq:X^k}) and (\ref{eq:basic-1}), we compute
\begin{align*}
Y_{j}^{k_{j}}X_{j+}&=(\beta_jY_j-\alpha_jI)X_{j+}=\beta_jX_{j+}X_{j-}^{-1}X_{j+}-\alpha_jX_{j+} \\
&=(s_j\beta_j-\alpha_j)X_{j+}-\beta_jX_{j-}\stackrel{(\ref{eq:omega1})}=\gamma_{j}X_{j+}-\beta_{j}X_{j-}, \\
Y_{j}^{-k_{j}}X_{j-}&=(-\beta_jY_j+\gamma_jI)X_{j-}=\gamma_jX_{j-}-\beta_jX_{j+},
\end{align*}
so for $j=1,2$,
\begin{align}
A_j&=(\gamma_{j}X_{j+}-\beta_{j}X_{j-})(\gamma_jX_{j-}-\beta_jX_{j+})  \label{eq:Aj-1}  \\
& =\gamma_{j}^{2}X_{j+}X_{j-}-\gamma_{j}\beta_{j}(X_{j+}^{2}+X_{j-}^{2})+\beta_{j}^{2}X_{j-}X_{j+},  \nonumber \\
& =(\gamma_j^2-\beta_j^2)X_{j+}X_{j-}+t(\beta_j-\gamma_j)\beta_j(X_{j+}+X_{j-})+(2\gamma_j-s_j\beta_j)\beta_jI;  \label{eq:Aj-2}
\end{align}
in the last line, we have used (\ref{eq:X^k}) and (\ref{eq:basic-2}).

Since
\begin{align*}
Y_3^{k_3}X_2^{-1}&=(\beta_3Y_3-\alpha_3I)X_2^{-1}=t\beta_3X_1X_2^{-1}-\beta_3X_1-\alpha_3X_2^{-1}, \\
Y_3^{-k_3}X_2&=(\omega_{-k_3}(s_3)Y_3-\omega_{-k_3-1}(s_3)I)X_2\stackrel{(\ref{eq:omega0})}=(\gamma_3I-\beta_3Y_3)X_2=\gamma_3X_2-\beta_3X_1,
\end{align*}
we have
\begin{align}
A_3&=(t\beta_3X_1X_2^{-1}-\beta_3X_1-\alpha_3X_2^{-1})(\gamma_3X_{2}-\beta_3X_{1})  \nonumber  \\
& =(\alpha_3-\gamma_3)\beta_3X_1X_2+t(\beta_3-\alpha_3)\beta_3(X_1+X_2)+(\alpha_3^2-\beta_3^2)I, \label{eq:A3}
\end{align}
where in the last line, (\ref{eq:X^k})--(\ref{eq:basic-2}) are applied.

The case when $t=0$ turns out to require a special treatment. To this end, we cite the result of \cite{Ch18-1} Section 3, stated in a different form:
\begin{lem} \label{lem:t=0}
Suppose $t=0$ and {\rm(\ref{eq:kappa=delta})} is satisfied. Then {\rm(\ref{eq:A})} holds if and only one of the following cases occurs:
\begin{itemize}
  \item $\gamma_j=\beta_j$, $j=1,2$ and $v_3^{2k_3}=-1$;
  \item $\gamma_j=-\beta_j$, $j=1,2$ and $\beta_3=0$;
  \item $\delta=0$, and there exists $\theta_j\in\mathbb{R}$ with $s_j=2\cos\theta_j$, $j=1,2,3$ and $\cos(2k_1+1)\theta_1=\cos(2k_2+1)\theta_2=\cos(2k_3\theta_3)\ne\pm1$.
\end{itemize}
\end{lem}

\begin{lem}
{\rm(a)} For $j=1,2$, $A_j=I$ if and only if $\gamma_j=\beta_j$ or $X_{j+}=X_{j-}^{-1}$.

{\rm(b)} If $\{X_1,X_2\}$ is regular, then $A_3=I$ if and only if $\beta_3=0$.
\end{lem}
\begin{proof}
(a) For $j=1,2$, from (\ref{eq:Aj-1}) we see that $A_j=I$ if and only if
$$\gamma_{j}X_{j+}-\beta_{j}X_{j-}=(\gamma_jX_{j-}-\beta_jX_{j+})^{-1}\stackrel{(\ref{eq:X-inverse})}=t(\gamma_j-\beta_j)I-\gamma_jX_{j-}+\beta_jX_{j+},$$
which is equivalent to $(\gamma_j-\beta_j)(X_{j+}-X_{j-}^{-1})=0$.

(b) By (\ref{eq:A3}) and the assumption on $(X_1,X_2)$, $A_3=I$ if and only if
$$(\alpha_3-\gamma_3)\beta_3=t(\alpha_3-\beta_3)\beta_3=\alpha_3^2-\beta_3^2-1=0.$$
Clearly this holds when $\beta_3=0$. Conversely, if $\beta_3\ne 0$, then
$\alpha_3=\gamma_3$ so that $(v_3-v_3^{-1})(v_3^{k_3}+v_3^{-k_3})=0$, and the last equality implies $s_3\alpha_3=2\beta_3$ so that $(v_3-v_3^{-1})(v_3^{k_3-1}+v_3^{1-k_3})=0$; since clearly $v_3\ne\pm1$, this is a contradiction.
\end{proof}

\begin{cor}
If $t\ne 0$ and $\gamma_j=\beta_j$ for some $j\in\{1,2\}$, then {\rm(\ref{eq:A})} holds if and only one of the following cases occurs:
\begin{enumerate}
  \item[\rm(i)] $\gamma_1=\beta_1,\gamma_2\ne\beta_2$, $X_3=X_1^{-1}$, $\beta_3=0$;
  \item[\rm(ii)] $\gamma_1\ne\beta_1,X_2=X_3^{-1}, \gamma_2=\beta_2$, $\beta_3=0$;
  \item[\rm(iii)] $\gamma_1=\beta_1, \gamma_2=\beta_2$, $\{X_1,X_2\}$ is regular, $\beta_3=0$;
  \item[\rm(iv)] $\gamma_1=\beta_1, \gamma_2=\beta_2$, $s_3=2$, $t\in\{\pm2\}$ and $X_1X_2=X_2X_1$.
\end{enumerate}
\end{cor}
\begin{proof}
Only case (iv) needs to be explained: if $\{X_1,X_2\}$ is not regular, then up to conjugacy we may assume
$X_1=\left(\begin{array}{cc} u & 1 \\ 0 & u^{-1} \end{array}\right)$, $X_2=\left(\begin{array}{cc} u^{\epsilon} & x \\ 0 & u^{-\epsilon} \end{array}\right)$ with $\epsilon\in\{\pm1\}$. From (\ref{eq:A3}) we can deduce $u\in\{\pm1\}$, $s_3=2$, and then clearly $X_1X_2=X_2X_1$.
\end{proof}

\begin{lem}
Suppose $t\ne 0$, $\gamma_j\ne \beta_j$ for $j=1,2$, $\beta_3\ne0$, and {\rm(\ref{eq:kappa=delta})} is satisfied. Let $\lambda=\tau/t$. Then {\rm(\ref{eq:A})} holds if and only
\begin{align}
(\lambda-2-s_j)\gamma_j&=(\sigma_1-s_j-\lambda)\beta_j, \qquad j=1,2, \label{eq:equivalence-main1} \\
(\sigma_1+2-2\lambda)\alpha_3&=(s_3^2-s_3\lambda+\sigma_1-2)\beta_3. \label{eq:equivalence-main2}
\end{align}
\end{lem}

\begin{proof}
By (\ref{eq:Aj-2}), for $j=1,2$,
\begin{align}
{\rm tr}(A_{j})& =(\gamma_{j}^{2}-\beta_{j}^{2})(t^2-s_j)+t(\beta_j-\gamma_j)\beta_j\cdot 2t+(2\gamma_{j}-s_j\beta_j)\beta_{j}\cdot 2 \nonumber \\
& =2-(s_j+2-t^2)(\gamma_{j}-\beta_{j})^2, \label{eq:tr-Aj} \\
{\rm tr}(A_jX_{j\pm}^{-1})&=(\gamma_j^2-\beta_j^2)t+t(\beta_j-\gamma_j)\beta_j(s_j+2)+(2\gamma_j-s_j\beta_j)\beta_jt=t, \label{eq:tr-AX+-}
\end{align}
and since ${\rm tr}(X_{j+}X_{j-}X_j^{-1})={\rm tr}(X_{j+}X_{j-}(tI-X_j))=\tau-t(1+s_j)$, we have
\begin{align}
{\rm tr}(A_jX_j^{-1})&=(\gamma_j^2-\beta_j^2)(\tau-t(1+s_j))+t(\beta_j-\gamma_j)(\sigma_1-s_j)+(2\gamma_j-s_j\beta_j)\beta_jt \nonumber \\
&=t(\gamma_j-\beta_j)\big((\lambda-2-s_j)\gamma_j-(\sigma_1-s_{j}-\lambda)\beta_j\big)+t. \label{eq:tr-AX}
\end{align}
By (\ref{eq:A3}),
\begin{align}
{\rm tr}(A_3)&=(\alpha_3-\gamma_3)\beta_3(t^2-s_3)+t(\beta_3-\alpha_3)\beta_3\cdot 2t+(\alpha_3^2-\beta_3^2)\cdot 2 \nonumber \\
& =2+(s_3-2)(s_3+2-t^2)\beta_3^2, \label{eq:tr-A3} \\
{\rm tr}(A_3X_{3\pm}^{-1})&=(\alpha_3-\gamma_3)\beta_3t+t(\beta_3-\alpha_3)\beta_3(2+s_3)+(\alpha_3^2-\beta_3^2)t=t, \label{eq:tr-AX3+-} \\
{\rm tr}(A_3X_3^{-1})&=(\alpha_3-\gamma_3)\beta_3t(\lambda-1-s_3)+t(\beta_3-\alpha_3)\beta_3(\sigma_1-s_3)+(\alpha_3^2-\beta_3^2)t \nonumber \\
& =t((s_3^2-s_3\lambda+\sigma_1-2)\beta_3-(\sigma_1+2-2\lambda)\alpha_3)\beta_3+t.  \label{eq:tr-AX3}
\end{align}

If (\ref{eq:A}) holds, then
\begin{align}
{\rm tr}(A_1)={\rm tr}(A_2)&={\rm tr}(A_3), \label{eq:equivalence1} \\
{\rm tr}(A_1X_j^{-1})={\rm tr}(A_2X_j^{-1})&={\rm tr}(A_3X_j^{-1}), \qquad j=1,2,3. \label{eq:equivalence2}
\end{align}
Due to the assumptions $t\ne 0$, $\gamma_1\ne \beta_1$, $\gamma_2\ne\beta_2$, and $\beta_3\ne0$, from (\ref{eq:tr-AX+-}), (\ref{eq:tr-AX}), (\ref{eq:tr-AX3+-}), (\ref{eq:tr-AX3}) we see that (\ref{eq:equivalence2}) is equivalent to (\ref{eq:equivalence-main1}) and (\ref{eq:equivalence-main2}).

Now suppose (\ref{eq:equivalence-main1}) and (\ref{eq:equivalence-main2}) hold (so that (\ref{eq:equivalence2}) is satisfied).

For $j=1,2$,
\begin{align}
(\sigma_1+2-2\lambda)\gamma_j&=(\sigma_1-s_j-\lambda)(\gamma_j-\beta_j), \label{eq:deduce1} \\
(\sigma_1+2-2\lambda)\beta_j&=(\lambda-2-s_j)(\gamma_j-\beta_j), \label{eq:deduce2}
\end{align}
hence
\begin{align*}
&(\sigma_1+2-2\lambda)^2\stackrel{(\ref{eq:omega2})}=(\sigma_1+2-2\lambda)^2(\gamma_j^2-s_j\gamma_j\beta_j+\beta_j^2)  \\
=\ &\big((\sigma_1-s_j-\lambda)^2-s_j(\sigma_1-s_j-\lambda)(\lambda-2-s_j)+(\lambda-2-s_j)^2\big)\cdot(\gamma_j-\beta_j)^2  \\
=\ &(t^{-2}\kappa(s_j+2)-\delta)(\gamma_j-\beta_j)^2, 
\end{align*}
where in the last line we use $s_j^3-\sigma_1s_j^2+\sigma_2s_j-\sigma_3=0$.
By (\ref{eq:kappa=delta}),
\begin{align*}
(\sigma_1+2-2\lambda)^2=t^{-2}\kappa(s_j+2-t^2)(\gamma_j-\beta_j)^2.  
\end{align*}
If $\sigma_1+2-2\lambda=0$, then (\ref{eq:deduce1}) and (\ref{eq:deduce2}) imply $s_1=s_2=\lambda-2, s_3=2$, so that $\delta=0$ and $\kappa=4t^2$, contradicting (\ref{eq:kappa=delta}) and the assumption $t\ne 0$.
Hence $\sigma_1+2-2\lambda\ne 0$, so that
\begin{align}
s_j\ne t^2-2, \qquad j=1,2. \label{eq:sj-ne}
\end{align}
Consequently,
\begin{align*}
{\rm tr}(A_j)-2=-(s_j+2-t^2)(\gamma_j-\beta_j)^2=-t^2\kappa^{-1}(\sigma_1+2-2\lambda)^2.
\end{align*}

Moreover,
\begin{align*}
1&\stackrel{(\ref{eq:omega2})}=\alpha_3^2-s_3\alpha_3\beta_3+\beta_3^2
\stackrel{(\ref{eq:equivalence-main2})}=\beta_3^2\left(\left(\frac{\sigma_1+2-2\lambda}{\sigma_1+2-2\lambda}\right)^2-s_3\frac{\sigma_1+2-2\lambda}{\sigma_1+2-2\lambda}+1\right) \\
&=\left(\frac{\beta_3}{\sigma_1+2-2\lambda}\right)^2\frac{\kappa}{t^2}(4-2t^2+t^2s_3-s_3^3),
\end{align*}
so that ${\rm tr}(A_3)-2=-t^2\kappa^{-1}(\sigma_1+2-2\lambda)^2$.
This establishes (\ref{eq:equivalence1}), which together with (\ref{eq:tr-A3}), (\ref{eq:tr-Aj}), (\ref{eq:sj-ne}) implies
\begin{align}
s_3\notin\{2,t^2-2\}. \label{eq:s3-ne}
\end{align}

The proof will be complete once $I, X_1^{-1}, X_2^{-1}, X_3^{-1}$ are shown to form a basis of $\mathcal{M}(2,\mathbb{C})$; then (\ref{eq:equivalence1}), (\ref{eq:equivalence2}) will imply that
${\rm tr}(A_1Z)={\rm tr}(A_2Z)={\rm tr}(A_3Z)$ for all $Z\in\mathcal{M}(2,\mathbb{C})$, forcing $A_1=A_2=A_3$.

Assume on the contrary that $I, X_1^{-1}, X_2^{-1}, X_3^{-1}$ are linearly dependent, so are $I, X_1, X_2, X_3$. Note that $I,X_1,X_2$ are linearly independent: otherwise $X_1=X_2$ or $X_1=X_2^{-1}$, which would respectively imply $s_3=2$ or $s_3=t^2-2$, either contradicting (\ref{eq:s3-ne}). So we have $X_3=aI+bX_1+cX_2$ for some $a,b,c$.
Then $t={\rm tr}(X_3)=2a+t(b+c)$, implying
\begin{align*}
\overline{X}_3=b\overline{X}_1+c\overline{X}_2, \qquad \text{with} \qquad \overline{X}_j=X_j-\frac{t}{2}I.  
\end{align*}
The following can be computed using (\ref{eq:basic-2}):
\begin{align*}
\overline{X}_1\overline{X}_2+\overline{X}_2\overline{X}_1&=pI,  \qquad \text{with} \qquad p=\frac{1}{2}t^2-s_3, \\
\overline{X}^2_j&=dI, \qquad \text{with}  \qquad d=\frac{1}{4}t^2-1.
\end{align*}
Consequently, ${\rm tr}(\overline{X}_1\overline{X}_2)=p$, ${\rm tr}(\overline{X}_1^2)={\rm tr}(\overline{X}_2^2)=2d$, and ${\rm tr}(\overline{X}_1^2\overline{X}_2)={\rm tr}(\overline{X}_1\overline{X}_2^2)=0.$
We can rewrite (\ref{eq:Aj-2}) as
\begin{align*}
A_1&=(\gamma_1^2-\beta_1^2)\overline{X}_{2}\overline{X}_{3}+\frac{t}{2}(\gamma_1-\beta_1)^2(\overline{X}_{2}+\overline{X}_{3})+f_1I \\
&=g_1\overline{X}_1\overline{X}_2+\frac{t}{2}(\gamma_1-\beta_1)^2(b\overline{X}_1+(c+1)\overline{X}_2)+h_1I, \\
A_2&=(\gamma_2^2-\beta_2^2)\overline{X}_{3}\overline{X}_{1}+\frac{t}{2}(\gamma_2-\beta_2)^2(\overline{X}_{3}+\overline{X}_{1})+f_2I \\
&=g_2\overline{X}_1\overline{X}_2+\frac{t}{2}(\gamma_2-\beta_2)^2((b+1)\overline{X}_1+c\overline{X}_2)+h_2I,
\end{align*}
and re-write (\ref{eq:A3}) as
\begin{align*}
A_3=g_3\overline{X}_{1}\overline{X}_{2}+\frac{t}{2}(\beta_3^2-\gamma_3\beta_3)(\overline{X}_{1}+\overline{X}_{2})+g_3I,
\end{align*}
where the $f_i$'s, $g_i$'s, $h_i$'s are coefficients that are irrelevant.

By (\ref{eq:equivalence1}), (\ref{eq:equivalence2}), ${\rm tr}(A_1\overline{X}_\ell)={\rm tr}(A_2\overline{X}_\ell)={\rm tr}(A_3\overline{X}_\ell)$, $\ell=1,2$. Hence
\begin{align*}
(\gamma_1-\beta_1)^2(2db+p(c+1))=(\gamma_2-\beta_2)^2(2d(b+1)+pc)&=(\beta_3^2-\gamma_3\beta_3)(2d+p), \\
(\gamma_1-\beta_1)^2(pb+2d(c+1))=(\gamma_2-\beta_2)^2(p(b+1)+2dc)&=(\beta_3^2-\gamma_3\beta_3)(2d+p),
\end{align*}
which lead to
\begin{align*}
2db+p(c+1)&=pb+2d(c+1), \\
2d(b+1)+pc&=p(b+1)+2dc.
\end{align*}
These force $p=2d$, so that $s_3=2$. But this contradicts (\ref{eq:s3-ne}).
\end{proof}

\section{On the A-polynomial}

For background on A-polynomial, see \cite{CCGLS94,LR03}.

For a knot $K\subset S^3$, choose a meridian-longitude pair $(\mathfrak{m},\mathfrak{l})$ of $K$. Let
$$\mathcal{R}_U(K)=\{\rho\in\mathcal{R}(K)\colon\rho(\mathfrak{m}),\rho(\mathfrak{l})\in\mathcal{UT}\},$$
and define $\xi=(\xi_1,\xi_2):\mathcal{R}_U(K)\to\mathbb{C}^2$ by setting $\xi_1(\rho)$ (resp. $\xi_2(\rho)$) to be the upper-left entry of $\rho(\mathfrak{m})$ (resp. $\rho(\mathfrak{l})$). As a known fact, each component of the Zariski closure $\mathcal{V}$ of ${\rm Im}(\xi)$ has dimension 0 or 1. The defining polynomial of the 1-dimensional part of $\mathcal{V}$ is called the {\it A-polynomial} $A_K$. The importance of this two-variable polynomial has at least two aspects. As shown in \cite{CCGLS94}, ``boundary slopes are boundary slopes", meaning that the slope of each side of the Newton polygon of $A_K$ equals the boundary slop of an incompressible surface in $E_K$.
The {\it AJ conjecture}, formulated in \cite{Ga04}, asserts a close relation between the recurrence polynomial of the colored Jones polynomials and $A_K$ (see also \cite{Ge02,Le06}); it is till now still widely open and attracts much attention.

\begin{rmk}
\rm The A-polynomial is notoriously difficult to compute. One reason is that $\rho(\mathfrak{l})$ is usually quite complicated, and no general law has been found to simplify the expression.
\end{rmk}

For the knot at hand, $K=P(2k_1+1,2k_2+1,2k_3)$, take $\mathfrak{m}=x_1$. The corresponding longitude can be found to be
\begin{align*}
\mathfrak{l}=(x_1x_2^{-1})^{-k_3}(x_3^{-1}x_1^{-1})^{k_2}(x_2x_3)^{-k_1-1}(x_1x_2^{-1})^{k_3}(x_2x_3)^{-k_1}(x_3^{-1}x_1^{-1})^{k_2+1}.
\end{align*}

Let $\rho$ be a representation of $\pi_1(E_K)$ as in Section 3. Then
\begin{align*}
\rho(\mathfrak{l})=Y_3^{-k_3}Y_2^{k_2}Y_1^{-k_1-1}Y_3^{k_3}Y_1^{-k_1}Y_2^{k_2+1}.
\end{align*}
Suppose $\rho\in\mathcal{R}_U(K)$, and let $u,w$ denote the upper-left entries of $X_1,\rho(\mathfrak{l})$, respectively, so that ${\rm tr}(\rho(\mathfrak{l}))=w+w^{-1}$.
It is easy to see that
\begin{align}
\rho(\mathfrak{l})=\frac{w-w^{-1}}{u-u^{-1}}X_1+\frac{uw^{-1}-wu^{-1}}{u-u^{-1}}I.  \label{eq:L1}
\end{align}
Let $B_1=Y_3^{-k_3-1}Y_2^{k_2}$, $B_2=Y_1^{-k_1-1}Y_3^{k_3}$, $B_3=Y_1^{-k_1}Y_2^{k_2+1}$, then $B_3=B_2B_1$ and $B_1B_2B_3=Y_3^{-1}\rho(\mathfrak{l})$.
Hence
\begin{align}
{\rm tr}(B_3)&\ =\ {\rm tr}(B_1B_2)={\rm tr}(Y_3^{-1}\rho(\mathfrak{l})B_3^{-1}) \nonumber \\
&\stackrel{(\ref{eq:L1})}=\frac{w-w^{-1}}{u-u^{-1}}{\rm tr}(X_2B_3^{-1})+\frac{uw^{-1}-wu^{-1}}{u-u^{-1}}{\rm tr}(X_2X_1^{-1}B_3^{-1}) \nonumber \\
&\ =\frac{w-w^{-1}}{u-u^{-1}}{\rm tr}(B_3X_2^{-1})+\frac{uw^{-1}-wu^{-1}}{u-u^{-1}}{\rm tr}(B_3);  \label{eq:deduce}
\end{align}
in the last line we have used (\ref{eq:relation1}) to deduce
\begin{align}
B_3X_1=X_2B_3,   \label{eq:B3X1=X2B3}
\end{align}
so that ${\rm tr}(X_2X_1^{-1}B_3^{-1})={\rm tr}(B_3^{-1})={\rm tr}(B_3)$.
Rewrite (\ref{eq:deduce}) as
\begin{align}
(w+1){\rm tr}(B_3X_2^{-1})=(u+u^{-1}w){\rm tr}(B_3).  \label{eq:u-w-0}
\end{align}

Similarly as in \cite{Ch18-2}, we focus on the ``hard" part $\overline{A}_K$ of $A_K$ which is contributed by $\mathcal{X}_3$.
Since
$$B_3=\beta_1\beta_2Y_1-\beta_1\gamma_2Y_3^{-1}-\gamma_1\beta_2I+\gamma_1\gamma_2Y_2,$$
we have
\begin{align*}
{\rm tr}(B_3X_2^{-1})&\ =\ t(\beta_1\beta_2-\beta_1\gamma_2-\gamma_1\beta_2+\gamma_1\gamma_2(s_1+s_2+1-\lambda))  \\
&\ =\ t((\sigma_1-\lambda-s_3)\gamma_1\gamma_2+(\gamma_1-\beta_1)(\gamma_2-\beta_2)) \\
&\stackrel{(\ref{eq:equivalence-main1})}=\frac{t\beta_1\beta_2}{(\lambda-2-s_1)(\lambda-2-s_2)}\Big(\prod\limits_{j=1}^3(\sigma_1-\lambda-s_j)+(\sigma_1+2-2\lambda)^2\Big) \\
&\stackrel{(\ref{eq:kappa=delta})}=\frac{t\beta_1\beta_2}{(\lambda-2-s_1)(\lambda-2-s_2)}(\sigma_1+2-\lambda-t^2)\kappa, \\
{\rm tr}(B_3X_1)&\ =\ t(\beta_1\beta_2(\lambda-1-s_3)-\beta_1\gamma_2-\gamma_1\beta_2+\gamma_1\gamma_2) \\
&\ =\ t((\lambda-2-s_3)\beta_1\beta_2+(\gamma_1-\beta_1)(\gamma_2-\beta_2)) \nonumber \\
&\stackrel{(\ref{eq:equivalence-main1})}=\frac{t\beta_1\beta_2}{(\lambda-2-s_1)(\lambda-2-s_2)}\Big(\prod\limits_{j=1}^3(\lambda-2-s_j)+(\sigma_1+2-2\lambda)^2\Big)  \\
&\stackrel{(\ref{eq:kappa=delta})}=\frac{t\beta_1\beta_2}{(\lambda-2-s_1)(\lambda-2-s_2)}(\lambda-t^2)\kappa,
\end{align*}
and then
\begin{align*}
{\rm tr}(B_3)&\ \stackrel{(\ref{eq:X-inverse})}=\ t^{-1}({\rm tr}(B_3X_1^{-1})+{\rm tr}(B_3X_1)) \\
&\stackrel{(\ref{eq:B3X1=X2B3})}=t^{-1}({\rm tr}(B_3X_2^{-1})+{\rm tr}(B_3X_1))  \\
&\ =\ \frac{\beta_1\beta_2}{(\lambda-2-s_1)(\lambda-2-s_2)}(\sigma_1+2-2t^2)\kappa.
\end{align*}
Thus (\ref{eq:u-w-0}) becomes
\begin{align*}
(w+1)t(\sigma_1+2-\lambda-t^2)=(u+u^{-1}w)(\sigma_1+2-2t^2).
\end{align*}

Then $\overline{A}_K$, as a polynomial in $u,w$, can be obtained by computing the multi-variable resultant of the following (remembering $t=u+u^{-1}$):
\begin{align}
(\lambda-2-s_j)\gamma_j&=(\sigma_1-s_j-\lambda)\beta_j, \qquad j=1,2,  \label{eq:AP1} \\
(\sigma_1+2-2\lambda)\alpha_3&=(s_3^2-s_3\lambda+\sigma_1-2)\beta_3, \label{eq:AP2} \\
t^2(\lambda^2-(\sigma_1+2)\lambda+\sigma_2+4)&=4+\sigma_3+2\sigma_2-\sigma_1^2, \label{eq:AP3}  \\
(w+1)t(\sigma_1+2-\lambda-t^2)&=(u+u^{-1}w)(\sigma_1+2-2t^2). \label{eq:AP4}
\end{align}


\begin{thebibliography}{}

\bibitem{ABL18}
C. Ashley, J.-P. Burelle, S. Lawton,
\textsl{Rank 1 character varieties of finitely presented groups}.
Geometriae Dedicata 192 (2018), no. 1, 1--19. \\
\url{https://link.springer.com/article/10.1007/s10711-017-0281-6}


\bibitem{Ch18-1}
H.-M. Chen,
\textsl{Trace-free ${\rm SL}(2,\mathbb{C})$-representations of Montesinos links}.
J. Knot Theor. Ramif. 27 (2018), no. 8, 1850050 (10 pages). \\
\url{https://www.worldscientific.com/doi/abs/10.1142/S0218216518500505}

\bibitem{Ch18-2}
H.-M. Chen,
\textsl{Character varieties of odd classical pretzel knots}.
Int. J. Math. 29 (2018), no. 9, 1850060 (15 pages). \\
\url{https://www.worldscientific.com/doi/abs/10.1142/S0129167X1850060X}


\bibitem{CCGLS94}
D. Cooper, M. Culler, H. Gillet, D. D. Long, P. B. Shalen,
\textsl{Plane curves associated to character varieties of 3-manifolds.}
Invent. Math. 118 (1994), 47--84. \\
\url{https://link.springer.com/article/10.1007%2FBF01231526?LI=true}

\bibitem{CS83}
M. Culler, P.B. Shalen,
\textsl{Varieties of group representations and splittings of 3-manifolds}.
Ann. Math. 117 (1983), no. 1, 109--146. \\
\url{https://www.jstor.org/stable/2006973?seq=1#metadata_info_tab_contents}


\bibitem{Ga04}
S. Garoufalidis,
\textsl{On the characteristic and deformation varieties of a knot}.
Proceedings of the Casson Fest, Geom. Topol. Monogr., vol. 7, Geom. Topol. Publ., Coventry, 2004, 291--309 (electronic). \\
\url{https://msp.org/gtm/2004/07/gtm-2004-07-012p.pdf}


\bibitem{Ge02}
R. Gelca,
\textsl{On the relations between the A-polynomial and the Jones polynomial}.
Proc. Amer. Math. Soc. 130 (2002), no. 4, 1235--1241.  \\
\url{https://www.ams.org/journals/proc/2002-130-04/S0002-9939-01-06157-3/S0002-9939-01-06157-3.pdf}


\bibitem{Go09}
W.M. Goldman,
\textsl{Trace coordinates on Fricke spaces of some simple hyperbolic surfaces}.
arXiv:0901.1404.
\url{https://arxiv.org/abs/0901.1404}


\bibitem{Le06}
T.T.Q. L\^e,
\textsl{The colored Jones polynomial and the A-polynomial of knots}.
Adv. Math. 207 (2006), no. 2, 782--804. \\
\url{https://www.sciencedirect.com/science/article/pii/S0001870806000156}


\bibitem{LR03}
D.D. Long, A.W. Reid,
\textsl{Integral points on character variety}.
Math. Ann. 325 (2003), 299--321. \\
\url{https://link.springer.com/article/10.1007%2Fs00208-002-0380-y?LI=true}












\end{thebibliography}
\end{document}